\documentclass[12pt]{amsart}
\usepackage[bookmarksnumbered, plainpages]{hyperref}
\usepackage{amsthm}
\usepackage{amsmath}
\usepackage{amsfonts}
\usepackage{fdsymbol}
\usepackage{color,xcolor,graphicx}
\usepackage{float}
\usepackage{booktabs}
\newtheorem{thm}{Theorem}[section]
\newtheorem{cor}[thm]{Corollary}

\theoremstyle{definition}

\newtheorem{exa}[thm]{Example}
\newtheorem{ques}[thm]{Question}

\theoremstyle{remark}

\title{Construction of  New Gyrogroups and the Structure of their Subgyrogroups}

\author{\bf S. Mahdavi, A. R. Ashrafi$^{\star}$   and M. A. Salahshour  }
\thanks{$^{\star}$ Corresponding author (Email: ashrafi@kashanu.ac.ir)}

\address{\textbf{ Soheila Mahdavi and Ali Reza Ashrafi:} Department of Pure Mathematics, Faculty of Mathematical Sciences, University of Kashan, Kashan 87317$-$53153, I. R. Iran}

\address{\textbf{Mohammad Ali Salahshour:} Department of Mathematics, Savadkooh Branch, Islamic Azad University, Savadkooh, I. R. Iran}

\begin{document}
\maketitle
\begin{abstract}
Suppose that $G$ is a groupoid with binary operation $\otimes$. The pair $(G,\otimes)$ is said to be a gyrogroup if the operation $\otimes$ has a left identity,   each element $a \in G$ has a left inverse and the gyroassociative law and the left loop property are satisfied in $G$. In this paper, a method  for constructing new gyrogroups from old ones is presented and the structure of subgyrogroups of these gyrogroups are also given. As a consequence of this work, five $2-$gyrogroups of order $2^n$, $n\geq 3$, are presented. Some open questions are also proposed.

\vskip 3mm


\vskip 3mm

\end{abstract}


\section{Introduction}
A groupoid $G$ with a binary operation $\oplus$ is called a \textit{gyrogroup} if the following hold:
\begin{itemize}
\item there exists an element $0 \in G$ such that for all $x \in G$, $0 \oplus x = x$;

\item for each $a \in  G$, there exists $b \in G$ such that $b \oplus a$ $= 0$;

\item there exists a function $gyr: G \times G \longrightarrow Aut((G,\oplus))$ such that for every $a, b, c \in G$, $a \oplus (b \oplus c) =  (a \oplus b) \oplus gyr[a,b]c$, where  $gyr[a,b]c = gyr(a,b)(c)$;

\item for each $a, b \in G$, $gyr[a,b]$ $=$ $gyr[a \oplus b, b]$.
\end{itemize}
Note that these axioms imply their right counterpart. It is easy to see that a group is a gyrogroup  if we define the gyroautomorphisms to be the identity automorphism. For every $a,b \in G$, the mapping $gyr[a,b]$ is called the gyroautomorphism generated by $a$ and $b$. The gyrogroup structure is a result of a pioneering work of Abraham Unbar in the study of Lorentz group \cite{6,7}.

Following Ferreira \cite[Sect. 4]{f}, suppose that $T$ is a gyrogroup and $H$ is a nonempty subset of $T$. $H$ is a subgyrogroup of $T$, written $H \leq G$, if $H$ is a gyrogroup under the operation inherited from $T$ and the restriction of $gyr[a,b]$, $a, b \in H$, to $H$ becomes an gyroautomorphism of $H$. It is merit to mention here that Ferreira used the term ``gyrosubgroup" and the term ``subgyrogroup" was first used in \cite{44}. The subgyrogroup $H$ is normal in $T$, written $H \unlhd T$, if it is the kernel of a gyrogroup homomorphism from $T$ to another gyrogroup \cite{44}. The subgyrogroup $H$ of $T$ is said to be an $L-$subgyrogroup, denoted by $H \leq_L G$, if $gyr[a,h](H) = H$, for all $a \in G$ and $h \in H$, see \cite[Definition 8]{2} for more details.

Suppose that $(K, \oplus)$ is a gyrogroup. The gyrogroup cooperation or coaddition $\boxplus$ is a second binary operation on $K$ defined as $a \boxplus b = a \oplus gyr[a,\ominus b]b$, for all $a, b \in K$.  It is well-known that $(K,\boxplus)$ is also a gyrogroup named \textit{cogyrogroup} of $(K,\oplus)$, see \cite[Theorem 2.14]{9} for details.

Suppose that $L$ is a gyrogroup, $A, B$ are subsets of $G$ and $a \in A, b \in B$. Following Suksumran \cite[Definition 3.13]{4},  $conj_{a}(b)=(a\oplus b)\boxminus a$ is called the conjugate of $b$ by $a$ and $conj_{a}(B)=\lbrace (a\oplus b)\boxminus a \vert b\in B \rbrace $ is named the conjugate of $B$ by $a$. For simplicity of our argument, we set $conj_{A}(B)$ $=$ $\lbrace conj_{a}(B) \mid a\in A \rbrace$.

For the sake of completeness, we mention here two interesting results of Suksumran. These are Theorem 31 and Proposition 38 in \cite{2}, respectively.

\begin{thm}\label{t0}
 Let $G$ be a gyrogroup and let $H$ be a subgyrogroup of $G$. Then the following hold:
\begin{enumerate}
\item $H \unlhd G$ if and only if the operation on the coset space $G/H$ given by $(a \oplus H) \oplus (b \oplus H) = (a \oplus b) \oplus H$ is well defined.

\item Suppose the following conditions are satisfied:
\begin{enumerate}
\item  $gyr[h,a] = id_G$, for all $h \in H$ and $a \in G$;

\item  $gyr[a,b](H) \subseteq H$, for all $a, b \in G$;

\item  $a \oplus H = H \oplus a$, for all $a \in G$.
\end{enumerate}
Then $H \unlhd G$.
\end{enumerate}
\end{thm}

Throughout this paper our notations are standard and can be taken mainly from \cite{1,8}. We refer the interested readers to consult the survey \cite{5} for a complete history of gyrogroups. We also refer to \cite{42,44} for subgyrogroups, gyrogroup homomorphisms and quotient gyrogroups.

\section{Main Results}

The aim of this section is to first construct a new class of finite gyrogroups. Then the main properties of this class of gyrogroups will be investigated. Suppose that $H$ is a group and $A$ is a set, $A \cap H = \emptyset$ and $A$ has the same size as $H$. Choose a bijective map $\varphi: H \longrightarrow A$ and set $G = A \cup H$. Define the binary operation $\otimes$ on $G$ as follows:
$$a \otimes b = \left\{ \begin{array}{ll} ab & a, b \in H\\ \varphi(\varphi^{-1}(a)b) & a\in A, b \in H\\ \varphi(a\varphi^{-1}(b)) & a \in H, b \in A\\ \varphi^{-1}(a) \varphi^{-1}(b) & a, b \in A\end{array}\right..$$
Then it is easy to see that $(G,\otimes)$ is a group under the operation $\otimes$. We now present some notations that help us to generalize this result to gyrogroups.

Suppose $(H^+,\oplus)$ is a gyrogroup, $H^{-}$ is a set disjoint from $H^+$ and $\varphi: H^+ \longrightarrow H^{-}$ is a bijective map.  An arbitrary element of $H^+$ is denoted by $x^+$, and define $x^- = \varphi(x^+)$,  $G = H^+ \cup H^{-}$ and
$$a^\varepsilon \otimes b^\delta = \left\{ \begin{array}{ll}
a^+ \oplus b^+ & (\varepsilon, \delta )=(+,+) \ or \ (-,-)\\
(a^+\oplus b^+)^- & (\varepsilon, \delta)=(+,-) \ or \ (-,+)
\end{array}\right.,$$ where $a, b \in G$.

We are now ready to state our first result:

\begin{thm} \label{t1}
$(G,\otimes)$ is a gyrogroup and the gyrator $gyr_G: G \times G \longrightarrow Aut(G)$ is defined as:
$$gyr_{G}[a^{\epsilon},b^{\delta} ](t^{\gamma} ) = \left\{
\begin{array}{ll}
gyr_{H^{+}}[a^{+},b^{+}](t^{+})  & \gamma = +\\
(gyr_{H^{+}}[a^{+},b^{+}](t^{+})) ^{-} &  \gamma = - \\
\end{array}\right..$$
\end{thm}

\begin{proof}
Suppose that $0^+$ is the identity of $H^+$. Then, it is easy to see that for each $a \in G$, $0^+ \otimes a = a$. For $x^\epsilon \in G$, we define
$$\oslash x^\epsilon = \left\{\begin{array}{cl} \ominus x^+ & \epsilon= + \\ (\ominus x^+)^- & \epsilon = -\end{array}\right..$$ Then by definition of $\otimes$, $x^\epsilon \otimes (\oslash x^\epsilon) = (\oslash x^\epsilon) \otimes x^\epsilon  = 0^+$. Here, the notation $\oslash x$ is used for the inverse of an arbitrary element $x \in G$. To prove  that $gyr_G[a,b] \in Aut(G)$, we first show that $gyr_G[a,b]$ is a gyrogroup homomorphism. We  have two separate cases as follows:
\begin{enumerate}
\item[1.]   $( \gamma, \lambda) = (+,+) \    or     \   (-,-)$
. In this case, by definition of gyrogroup homomorphism,
\begin{eqnarray*}
gyr_{G}[a^{\epsilon},b^{\delta}](x^{\gamma}\otimes y^{\lambda})&=&gyr_{G}[a^{\epsilon},b^{\delta}](x^{+}\oplus y^{+})\\
&=&gyr_{H^{+}}[a^{+},b^{+}](x^{+}\oplus y^{+})\\
&=&gyr_{H^{+}}[a^{+},b^{+}]x^{+}\oplus gyr_{H^{+}}[a^{+},b^{+}]y^{+}\\
&=&gyr_{G}[a^{\epsilon},b^{\delta}]x^{\gamma}\otimes gyr_{G}[a^{\epsilon},b^{\delta}]y^{\lambda}.
\end{eqnarray*}

\item[2.] $( \gamma, \lambda) = (+,-)  \ or \ (-,+)$. In this case, we have
\begin{eqnarray*}
gyr_{G}[a^{\epsilon},b^{\delta}](x^{\gamma}\otimes y^{\lambda})&=&gyr_{G}[a^{\epsilon},b^{\delta}]((x^{+}\oplus y^{+})^{-})\\
&=&(gyr_{H^{+}}[a^{+},b^{+}](x^{+}\oplus y^{+}))^{-}\\
&=&(gyr_{H^{+}}[a^{+},b^{+}]x^{+}\oplus gyr_{H^{+}}[a^{+},b^{+}]y^{+})^{-}\\
&=&gyr_{G}[a^{\epsilon},b^{\delta}]x^{\gamma}\otimes gyr_{G}[a^{\epsilon},b^{\delta}]y^{\lambda}.
\end{eqnarray*}

\end{enumerate}

This proves that $gyr_G[a,b]$ is a gyrogroup homomorphism. A case by case investigation shows that  $gyr_G[a,b]$ is one-to-one and so it is an automorphism of the groupoid $(G,\otimes)$. In what follows, the gyroassociative law is investigated in two different cases. To do this, we assume that  $a^{\epsilon}, b^{\delta}$ and $c^{\gamma}$ are arbitrary elements of $G$. We first assume that  $\gamma=+$ . Then,
\begin{itemize}
\item[1.] $(\epsilon,\delta) = (+,+)$ or $(-,-)$. A simple calculation shows that
\begin{eqnarray*}
a^{\epsilon}\otimes (b^{\delta}\otimes c^{+})&=&a^{+}\oplus (b^{+} \oplus c^{+})\\
&=&(a^{+}\oplus b^{+})\oplus gyr_{H^{+}}[a^{+},b^{+}]c^{+}\\
&=&(a^{+}\oplus b^{+})\otimes gyr_{H^{+}}[a^{+},b^{+}]c^{+}\\
&=&(a^{\epsilon}\otimes  b^{\delta})\otimes gyr_{G}[a^{\epsilon},b^{\delta}]c^{+}.
\end{eqnarray*}

\item[2.]  $(\epsilon,\delta)=(-,+)$  or $(+,-)$. In this case, we have:
\begin{eqnarray*}
a^{\epsilon}\otimes (b^{\delta}\otimes c^{+})&=&(a^{+}\oplus (b^{+}\oplus c^{+}))^{-}\\
&=&((a^{+}\oplus b^{+})\oplus gyr_{H^{+}}[a^{+},b^{+}]c^{+})^{-}\\
 &=&(a^{+}\oplus b^{+})^{-}\otimes gyr_{H^{+}}[a^{+},b^{+}]c^{+} \\
&=&(a^{\epsilon}\otimes b^{\delta})\otimes gyr_{G}[a^{\epsilon},b^{\delta}]c^{+}.
\end{eqnarray*}
\end{itemize}

Next we assume that $\gamma=-$ . Then
\begin{itemize}
\item[1.] $(\epsilon,\delta) = (+,+)$ or $(-,-)$. A simple calculation shows that
\begin{eqnarray*}
a^{\epsilon}\otimes(b^{\delta}\otimes c^{-})&=&(a^{+}\oplus (b^{+}\oplus c^{+}))^{-}\\
&=&((a^{+}\oplus b^{+})\oplus gyr_{H^{+}} [a^{+},b^{+}]c^{+})^{-}\\
  &=&(a^{+}\oplus b^{+}) \otimes (gyr_{H^{+}}[a^{+},b^{+}]c^{+})^{-}\\
  &=&(a^{\epsilon}\otimes b^{\delta})\otimes gyr_{G}[a^{\epsilon},b^{\delta}]c^{-}.
\end{eqnarray*}

\item[2.]  $(\epsilon,\delta)=(-,+)$  or $(+,-)$. In this case, we have:
\begin{eqnarray*}
a^{\epsilon}\otimes(b^{\delta}\otimes c^{-})&=&a^{+}\oplus (b^{+}\oplus c^{+}) \\
&=&(a^{+}\oplus b^{+})\oplus gyr_{H^{+}}[a^{+},b^{+}]c^{+}\\
 &=&(a^{+}\oplus b^{+})^{-}\otimes(gyr_{H^+}[a^{+},b^{+}]c^{+})^{-}\\
 &=&(a^{\epsilon}\otimes b^{\delta})\otimes gyr_{G}[a^{\epsilon},b^{\delta}]c^{-}.
\end{eqnarray*}
\end{itemize}

This proves that the gyroassociative lave is valid. To complete the proof, we have to prove the left loop property.
To de this, we first assume that  $\gamma=+$ . Then,
\begin{eqnarray*}
gyr_{G}[a^{\epsilon},b^{\delta}](t^{+})&=&gyr_{H^{+}}[a^{+},b^{+}]t^{+}\\
&=&gyr_{H^{+}}[a^{+}\oplus b^{+},b^{+}]t^{+}\\
&=&gyr_{G}[a^{\epsilon}\otimes b^{\delta},b^{\delta}]t^{+},
\end{eqnarray*}
as desired. Next we assume that $\gamma=-$ . In these cases, we have
\begin{eqnarray*}
gyr_{G}[a^{\epsilon},b^{\delta}](t^{-})&=& (gyr_{H^{+}}[a^{+} ,b^{+}]t^{+})^{-}\\
&=&(gyr_{H^{+}}[a^{+}\oplus b^{+},b^{+}]t^+)^{-}\\
&=&gyr_{G}[a^{\epsilon}\otimes b^{\delta},b^{\delta}]t^{-}
\end{eqnarray*}
which completes the proof.
\end{proof}

From now on, $(H^+,\oplus)$ is an arbitrary gyrogroup and $(G,\otimes)$ is its associated gyrogroup constructed in Theorem \ref{t1}.

\begin{cor}
Let $H^+$ and $G$ be gyrogroups as in Theorem \ref{t1}. Then $H^{-} = 0^{-} \otimes H^{+}$ and $G=H^{+}\cup (0^{-}\otimes H^{+})$, where $0^{-}=\varphi(0^{+})$.
\end{cor}

\begin{proof}
By Theorem \ref{t1},
\begin{eqnarray*}
0^{-}\otimes H^{+}&=&\lbrace 0^{-}\otimes h^{+}\vert h^{+}\in H^{+}\rbrace\\
&=& \lbrace (0^{+}\oplus h^{+})^{-}\vert h^{+}\in H^{+}\rbrace\\
&=& \lbrace  (h^{+})^{-}\vert h^{+}\in H^{+}\rbrace\\
&=& H^{-}.
\end{eqnarray*}
Thus $G=H^{+}\cup H^{-}=H^{+}\cup (0^{-}\otimes	 H^{+})$.
\end{proof}

\begin{cor}
If  $(H^+,\oplus)$ is gyrocommutative then $(G,\otimes)$ is also gyrocommutative.
\end{cor}

\begin{proof}
Suppose $a^\epsilon$ and $b^\delta$ are arbitrary in $G$. We consider two separate cases as follows:
\begin{itemize}
\item[1.] $(\epsilon,\delta) = (+,+)$ or $(-,-)$. A simple calculation shows that
\begin{eqnarray*}
gyr_{G}[a^{\epsilon},b^{\delta}](b^{\delta}\otimes a^{\epsilon})&=&gyr_{G}[a^{\epsilon} ,b^{\delta}](b^{+}\oplus a^{+})\\
&=&gyr_{H^{+}}[a^{+} ,b^{+}](b^{+}\oplus a^{+})\\
&=&a^{+}\oplus b^{+}\\
&=&a^{\epsilon}\otimes b^{\delta}.
\end{eqnarray*}

\item[2.]  $(\epsilon,\delta)=(-,+)$  or $(+,-)$. In this case, we have:
\begin{eqnarray*}
gyr_{G}[a^{\epsilon},b^{\delta}](b^{\delta}\otimes a^{\epsilon})&=&gyr_{G}[a^{\epsilon} ,b^{\delta}]((b^{+}\oplus a^{+})^{-})\\
&=&(gyr_{H^{+}}[a^{+},b](b^{+}\oplus  a^{+}))^{-}\\
&=&(a^{+}\oplus b^{+})^{-}\\
&=&a^{\epsilon}\otimes b^{\delta}.
\end{eqnarray*}
\end{itemize}
This completes the proof.
\end{proof}

A nondegenerate  gyrogroup is a gyrogroup which is not a group. A simple calculation by Gap \cite{45} shows that there is no nondegenerate  gyrogroup of order $\leq 7$. On the other hand, our calculations recorded in Tables 1 and 2 show that the quasigroups $K(1)$, $L(1)$, $M(1)$, $N(1)$ and $O(1)$ are gyrogroups of order $8$, but we do not have an efficient algorithm to construct all gyrogroups of this order. So, it is natural to ask the following question:

\begin{ques}
How many  gyrogroups of order $8$ are there up to isomorphism?
\end{ques}

In Tables 1 and 2, the Cayley tables and its associated gyration tables of the gyrogroups $(K(1),\oplus_K)$, $(L(1),\oplus_L)$, $(M(1),\oplus_M)$, $(N(1),\oplus_N)$ and $(O(1),\oplus_O)$ of order $8$ are given. For simplicity of our argument,  the underlying set of each gyrogroup is assumed to be  $\{ 0, 1, 2, 3, 4, 5, 6, 7\}$. In these tables, $A = (4,5)(6,7)$, $B = (2,3)(4,5)$, $C = (4,5)$, $D = (2,3)(6,7)$ and $E = (4,5)(6,7)$ are automorphisms of the quasigroups $K(1)$, $L(1)$, $M(1)$, $N(1)$ and $O(1)$, respectively.

Suppose that $\boxplus$ and $\boxtimes$ are coadditions of $\oplus$ and $\otimes$ in the gyrogroups $H^+$ and $G$, respectively. It is usual to use the notations $\boxminus x$ and $\boxslash$ for the inverse of $x$ with respect to $\boxplus$ and $\boxtimes$, respectively. For each $a^\epsilon, b^\delta \in G$, we have:
\[ a^{\epsilon}\boxslash b^{\delta} =
\begin{cases}
a^{+}\boxminus b^{+}    & (\epsilon,\delta)=(+,+) \ or \ (-,-),\\
 (a^{+}\boxminus b^{+})^{-}   & (\epsilon , \delta)=(+,-) \ or \ (-,+).\\
\end{cases}
\]

\begin{exa} \label{e1}
Suppose that $R$ is a gyrogroup, $S$ is a normal subgyrogroup of $R$ and $a\in R$. By \cite[Proposition 39]{2},  $a \oplus S = S \oplus a$ and so $(a \oplus S) \boxminus a = (S \oplus a) \boxminus a = S$. Hence $S=conj_{a}(S)$. However, the converse is not generally true. To see this, we consider the gyrogroup $K(1)$ of order $8$ introduced in Table 1 and set $P = \{ 0, 2\}$. It is easy to see that $P=conj_{a}(P)$. If $P$ is normal in $K(1)$ then by \cite[Theorem 32]{2}, $( a\oplus b)\oplus P = a \oplus (P \oplus b) = (a \oplus P) \oplus b$, where $a$ and $b$ are arbitrary elements of $K(1)$. Suppose that $a=5$ and $b=6$. Then $(5\oplus 6) \oplus \lbrace 0,2 \rbrace$ $=$ $\lbrace 3,1 \rbrace \neq (5 \oplus \lbrace 0,2 \rbrace) \oplus 6$ $=$ $\lbrace 3,0\rbrace$. This proves that  $P$ is not normal in $K(1)$.
\end{exa}

\begin{ques}
Find a condition on the gyrogroup $L$ such that all subgyrogroups $T$ of $L$ satisfy the following condition:

\centerline{$\forall a \in L$, $conj_a(T) = T$ $\Leftrightarrow$ $T \unlhd L$.}
\end{ques}

Following Suksumran \cite{41}, let $L$ be a gyrogroup and let $a, b \in L$. The commutator $[a,b]$ is defined as
$[a,b] = \ominus (a \oplus b) \oplus gyr[a,b](b \oplus a)$ and the derived subgroup $L^\prime$ is the subgyrogroup generated by all commutators.  The author of the mentioned paper also noted that unlike the situation in group theory, it is still an open problem whether the derived subgyrogroup of a gyrogroup $L$, is normal
in $L$. Suppose that $L$ is a group and $T$ is a subgroup of $L$. It is an elementary fact that if $L^\prime \leq T$ then $T$ is normal in $L$. This is a generalization of the normality of derived subgroup. Again consider the subgyrogroup $P$ of $K$ presented in Example \ref{e1}. Then $K^\prime = \{ 0\} \leq P$, but $P$ is not normal in $K$. Hence the following question is natural:

\begin{ques}
Find a condition on the gyrogroup $L$ such that for all subgyrogroups $T$ of $L$, $L^\prime \leq T$ implies that $T \unlhd L$.
\end{ques}

\begin{thm} \label{t3}
With notations of Theorem  2.1, $conj_{a^{+}}(G)$ $=$ $conj_{a^{-}}(G)$ $=$ $conj_{a^{+}}(H^{+})$ $\cup$ $(conj_{a^{+}}(H^{+}))^{-}.$ In particular, $\vert conj_{a^{+}}(G) \vert = \vert conj_{a^{-}}(G) \vert = 2\vert conj_{a^{+}}(H^{+}) \vert$.
\end{thm}
 \begin{proof}
 Suppose   $b^{\delta}$ is an arbitrary element of $G$. Then, we have

\begin{eqnarray*}
conj_{a^{+}}(G)&=&\lbrace (a^{+}\otimes b^{\delta})\boxslash a^{+}\ \vert \ b^{\delta}\in G \rbrace \\
&=&\lbrace (a^{+}\otimes b^{+})\boxslash a^{+}\ \vert \ b^{+}\in G \rbrace \cup \lbrace (a^{+}\otimes b^{-})\boxslash a^{+}\ \vert\  b^{-}\in G \rbrace\\
&=&\lbrace (a^{+}\oplus b^{+})\boxslash a^{+}\ \vert\  b^{+}\in H^+ \rbrace \cup \lbrace (a^{+}\oplus b^{+})^{-}\boxslash a^{+}\ \vert\ b^{-}\in H^- \rbrace\\
&=&\lbrace (a^{+}\oplus b^{+})\boxminus a^{+}\ \vert\  b^{+}\in H^+ \rbrace \cup \lbrace ((a^{+}\oplus b^{+})\boxminus a^{+})^{-}\ \vert\ b^{+}\in H^+ \rbrace\\
&=& conj_{a^{+}}(H^{+}) \cup (conj_{a^{+}}(H^{+}))^{-}.
\end{eqnarray*}

\begin{eqnarray*}
conj_{a^{-}}(G)&=&\lbrace (a^{-}\otimes b^{\delta})\boxslash a^{-}\ \vert\ b^{\delta}\in G \rbrace \\
&=&\lbrace (a^{-}\otimes b^{+})\boxslash a^{-}\ \vert\ b^{+}\in G \rbrace \cup \lbrace (a^{-}\otimes b^{-})\boxslash a^{-}\ \vert\ b^{-}\in G \rbrace\\
&=&\lbrace (a^{+}\oplus b^{+})^-\boxslash a^{-}\ \vert\ b^{+}\in H^+ \rbrace \cup \lbrace (a^{+}\oplus b^{+})\boxslash a^{-}\ \vert\ b^{-}\in H^- \rbrace\\
&=&\lbrace (a^{+}\oplus b^{+})\boxminus a^{+}\ \vert\ b^{+}\in H^+ \rbrace \cup \lbrace ((a^{+}\oplus b^{+})\boxminus a^{+})^-\ \vert\ b^{+}\in H^+ \rbrace\\
&=& conj_{a^{+}}(H^{+}) \cup (conj_{a^{+}}(H^{+}))^{-}
\end{eqnarray*}

This shows that $conj_{a^{+}}(G)=conj_{a^{-}}(G)=conj_{a^{+}}(H^{+}) \cup (conj_{a^{+}}(H^{+}))^{-}$. Since
$conj_{a^{+}}(H^{+})$ $\cap$ $(conj_{a^{+}}(H^{+}))^{-}=\emptyset$ and $\vert conj_{a^{+}}(H^{+}) \vert$ $=$ $\vert   (conj_{a^{+}}(H^{+}))^{-} \vert $,  $\vert conj_{a^{+}}(G) \vert =2\vert conj_{a^{+}}(H^{+}) \vert$ which completes the proof.
\end{proof}

\begin{thm}
With notations of Theorem \ref{t1}, $(H^{+})^{'}=G^{'}$
\end{thm}
\begin{proof}
Suppose that $a^{\epsilon}$ and $b^{\delta}$ are arbitrary elements of $G$. We first prove that a commutator of two elements in $G$ is equal to a commutator of some elements in $H^+$. To do this, the following two cases are considered:
\begin{itemize}
\item[1.] $(\epsilon,\delta) = (+,+)$ or $(-,-)$. Then,
\begin{eqnarray*}
[a^{\epsilon},b^{\delta}]_{G}&=&\oslash (a^{\epsilon}\otimes b^{\delta})\otimes gyr_{G}[a^{\epsilon},b^{\delta}](b^{\delta}\otimes a^{\epsilon})\\
&=&\oslash(a^{+}\oplus b^{+})\otimes gyr_{G}[a^{\epsilon},b^{\delta}](b^{+}\oplus a^{+})\\
&=&\ominus(a^{+}\oplus b^{+})\otimes gyr_{H^{+}}[a^{+},b^{+}](b^{+}\oplus a^{+})\\
&=&\ominus(a^{+}\oplus b^{+})\oplus gyr_{H^{+}}[a^{+},b^{+}](b^{+}\oplus a^{+})\\
&=&[a^{+},b^{+}]_{H^{+}}.
\end{eqnarray*}

\item[2.] $(\epsilon,\delta) = (-,+)$ or $(+,-)$. In this case,
\begin{eqnarray*}
[a^{\epsilon},b^{\delta}]_{G}&=&\oslash (a^{\epsilon}\otimes b^{\delta})\otimes gyr_{G}[a^{\epsilon},b^{\delta}](b^{\delta}\otimes a^{\epsilon})\\
&=&\oslash ((a^{+}\oplus b^{+})^{-})\otimes gyr_{G}[a^{\epsilon},b^{\delta}]((b^{+}\oplus a^{+})^{-})
\\
&=&(\ominus(a^{+}\oplus b^+))^{-}\otimes (gyr_{H^{+}}[a^{+},b^{+}](b^{+}\oplus a^{+}))^{-}\\
&=&\ominus (a^{+}\oplus b^{+})\oplus gyr_{H^{+}}[a^{+},b^{+}](b^{+}\oplus a^{+})\\
&=&[a^{+},b^{+}]_{H^{+}}.
\end{eqnarray*}
\end{itemize}
This proves that $G^{'} = \langle [x,y] \mid x, y \in G \rangle = \langle [x,y] \mid x, y \in  H^+ \rangle = (H^+)^{'}$, proving the lemma.
\end{proof}

\begin{thm}
Suppose that $G$ and $H^+$ are finite gyrogroups as in Theorem \ref{t1}. Then  $H^{+} \unlhd  G$.
\end{thm}

\begin{proof}
Note that $gyr_{G}[a^{\epsilon},b^{\delta}](H^{+})=\lbrace gyr_{G}[a^{\epsilon},b^{\delta}]c^{+}\vert c^{+}\in H^{+}\rbrace$.  By Theorem \ref{t1}, $gyr_{G}[a^{\epsilon},b^{\delta}]c^{+}=gyr_{H^{+}}[a^{+},b^{+}]c^{+}$.
On the other hand, $gyr_{H^{+}}[a^{\epsilon},b^{\delta}] \in Aut(H^{+})$ and hence $gyr_{G}[a^{\epsilon},b^{\delta}](H^{+})$ $\subseteq H^{+}$. Since $[G:H^{+}]=2$, by \cite[Theorem 4.5]{41}, we conclude that $H^{+}\unlhd G$.
\end{proof}

\begin{thm} \label{t9_0}
 Suppose that $H^{+}$  and $G$ are gyrogroups as in Theorem \ref{t1}. If  $B$ is a subgyrogroup of $G$ such that $B\nsubseteq H^+$, then the following hold:
\begin{enumerate}
\item  There exists $A^+\leq H^+$ and $L^- \subseteq H^-$ such that $B = A^+ \cup L^-$;

\item  $ A^+ \cap L^{+} =\emptyset$ or $L^+ = A^+$, where $L^+ = \varphi^{-1}(L^-)$;

\item  $A^+ \cup L^+ \leq H^+$;

\item  $|A^+|=|L^-|$.
\end{enumerate}
\end{thm}
\begin{proof}
\begin{enumerate}
\item It is easy to see that $B \not\subseteq H^-$. Set $A^+ = B \cap H^{+}$   and    $L^- = B \cap H^{-}$. Assume  that $B\leq G$,  $B\nsubseteq H^+$ and $B \nsubseteq H^-$ then $A^+$ and $L^-$ are not empty sets. Also $A^+\leq H^+$, $L^- \subseteq H^-$ and $B = A^+ \cup L^-$.
\item Suppose  $ A^+ \cap L^{+} \ne \emptyset$. We prove that $L^+ = A^+$. First we show $ A^+ \cap L^{+}$ is a subgyrogroup of $H^+$. To do this, we choose two elements $a^+, b^+ \in  A^+ \cap L^{+}$.  Since $ A^+ \leq H^+$, $a^+ \oplus b^+ \in A^+$. On the other hand, $a^+ \in A^+ \subseteq B$, $b^- \in L^- \subseteq B$ and $B  \leq G$, it can easily see that $(a^+ \oplus b^+)^-=a^+ \otimes b^- \in B\cap H^- = L^-$. This proves that $a^+ \oplus b^+ \in L^+$ and so $a^+ \oplus b^+ \in A^+ \cap L^+$. We now apply \cite[Proposition 22]{2} to deduce that $ A^+ \cap L^{+} \leq H^+$.

Now, we have to show that  $L^+ = A^+$. Suppose that $b^+ \in L^+$ is arbitrary. Since  $ A^+ \cap L^{+} \leq H^+$, $0^+ \in A^+ \cap L^{+}$. Thus, $0^-, b^- \in L^- \subseteq B$. But $B \leq G$, so $0^- \otimes b^- = 0^+ \oplus b^+ = b^+ \in B = A^+ \cup L^-$. This means that $b^+ \in A^+$ and therefore $L^+ \subseteq A^+$. To prove the converse, let $b^+ \in A^+$. Note that $0^+ \in A^+ \cap L^{+}$ and so $0^-  \in L^- \subseteq B$. But $0^- \otimes b^+ = (0^+ \oplus b^+)^- = b^- \in B = A^+ \cup L^-$. This shows that $b^+ \in L^+$, i.e., $A^+ \subseteq L^+$. Therefore, $L^+ = A^+$.

\item Define $K=A^+ \cup L^{+}$.  We will prove that $(K,\oplus)$ is a subgyrogroup of $H^{+}$. To do this, we choose the elements $a^{+}$ and $b^{+}$ in $K$. If $a^{+}, b^{+} \in A^+$ then $a^{+} \oplus  b^{+}\in A^+ \subseteq K$, since $A^+\leq H^+$ as desired.  If $a^{+}, b^{+} \in L^{+}$, then $a^{-}, b^{-} \in L^- = B \cap H^- \subseteq B$. This shows that $a^{-}\otimes b^{-} \in B$. On the other hand, $a^{-} \otimes b^{-} = a^+ \oplus b^+ \in B \cap H^+ = A^+ \subseteq K$. We now assume that  $a^{+}\in A^+$ and $b^{+}\in L^{+}$. It is clear that $(a^{+}\oplus b^{+})^{-}\in H^{-}$. Since $a^{+}\in A^+ \subseteq  B$ and $b^{-}\in L^- \subseteq B$,  $(a^{+}\oplus b^{+})^-=a^+ \otimes b^- \in B$. Thus, $ (a^{+} \oplus b^{+})^-\in B\cap H^{-}=L^-$ and so $a^{+}\oplus b^{+} \in L^{+}\subseteq K$.  Finally, suppose that  $a^{+}\in L^+$ and $b^{+} \in A^{+}$. A similar argument proves that $a^{+}\oplus b^{+} \in L^{+}\subseteq K$. Therefore, by \cite[Proposition 22]{2},  $K$ is a subgyrogroup of $H^{+}$.

\item By (1) $A^+$ and $L^-$ are not empty sets. Let $A^+ = \{ a_1^+, \ldots, a_n^+\}$. Since $L^-\ne \emptyset$, there exists an element $x^- \in L^-$. So $ x^-\otimes a_i^+ = ( x^+\oplus a_i^+)^- \in B\cap H^-  = L^-$. This shows that $ x^-\otimes A^+ = \{  x^-\otimes a_1^+, \ldots, x^-\otimes a_n^+ \} \subseteq L^-$. On the other hand, by the left cancelation law the elements of $ x^-\otimes A^+$ are distinct thus $|A^+| \leq |L^-|$. A similar argument proves that if $L^- = \{ b_1^-, \ldots, b_m^-\}$, then for each $y^- \in L^-$, 	$y^- \otimes L^- = \{ y^- \otimes b_1^-, \ldots, y^- \otimes b_m^- \}  \subseteq A^+$. Therefore, $|L^-| \leq |A^+|$, which proves the result.
\end{enumerate}
Hence the result.
\end{proof}

\begin{thm} \label{t9}
 Suppose that $H^{+}$  and $G$ are gyrogroups of Theorem \ref{t1}. A non-empty subset  $B$ of $G$ is a subgyrogroup of $G$ if and only if one of the following hold:
\begin{enumerate}
\item  $B\leq H^{+}$;

\item  There exists $A^+\leq H^+$ and $L^- \subseteq H^-$ such that $B = A^+ \cup L^-$ with the property that for each $x, y \in L^-$, we have $x \otimes y \in A^+$. Also $ A^+ \cap L^{+} = \emptyset$ and $ A^+ \cup L^{+} \leq H^+$ that  $L^+ = \varphi^{-1}(L^-)$;

\item  $B=A^+\cup A^-$ such that  $A^+\leq H^+$  and $A^-=\varphi(A^+)$.

\end{enumerate}
\end{thm}


\begin{proof}
Let $B$ be a subgyrogroup of $G$. If  $B\subseteq H^+$, then there is nothing to prove. Hence it is enough to consider the case where $B \not\subseteq H^+$. By Theorem \ref{t9_0}, the conditions (2) and (3) are true.

Conversely, we assume that $B$ is a nonempty subset of $G$ which satisfies the conditions of the theorem.
\begin{enumerate}
\item If $B \leq H^+$ then there is nothing to prove.

\item Suppose that $B$ satisfy the condition $(2)$. By \cite[Proposition 22]{2}, it is enough to prove that $B$ is closed under the operation of $G$.
We choose arbitrary elements $a^\epsilon$ and $b^\delta$ from $B$. We first assume that $a^\epsilon, b^\delta \in A^+$. Then  $\epsilon = \delta = +$ and since $A^+ \leq H^+$, $a^\epsilon \otimes b^\delta =  a^+ \otimes b^+=a^+ \oplus b^+ \in A^+  \subseteq B$, as desired. Furthermore, if $a^\epsilon, b^\delta \in L^-$, then by our assumption $a^\epsilon \otimes b^\delta \in A^+ \subseteq B$, which is our goal.
We now assume that $a^+ \in A^+$ and $b^- \in L^-$. Since $A^+ \cup L^+$ is a subgyrogroup of $H^+$,  $a^+ \oplus b^+ \in A^+ \cup L^+$. We claim that $a^+ \oplus b^+ \in L^+$; otherwise, $a^+ \oplus b^+ \in A^+$ and $b^+ = \ominus a^+ \oplus (a^+ \oplus b^+) \in A^+ \cap L^+$, a contradiction. Thus, $a^+ \oplus b^+ \in L^+$ and hence $a^\epsilon \otimes b^\delta$ $=$ $a^+ \otimes b^-$ $=$ $(a^+ \oplus b^+)^- \in L^- \subseteq B$, as imagined.
For the final case, we assume that  $a^- \in L^-$ and $b^+ \in A^+$. Since $A^+ \cup L^+$ is a subgyrogroup of $H^+$, $a^+ \oplus b^+ \in A^+ \cup L^+$. We claim that $a^+ \oplus b^+ \in L^+$. Otherwise, $a^+ \oplus b^+ \in A^+$ and $a^+ = (a^+ \oplus b^+) \boxminus b^+ \in A^+ \cap L^+$, a contradiction. Thus, $a^+ \oplus b^+ \in L^+$ and hence  $a^- \otimes b^+$ $=$ $(a^+ \oplus b^+)^- \in L^- \subseteq B$, which completes this part of the task.

\item In this case, we have to prove that $B = A^- \cup A^+$ is closed under the operation $\otimes$. To see this, we choose arbitrary elements $a^\epsilon$ and $b^\delta$ in $B$. We first assume that $a^+, b^+ \in A^+$. Since $A^+ \leq H^+$,  $a^+ \otimes b^+ = a^+ \oplus b^+ \in A^+  \subseteq B$, as desired. If $a^-, b^- \in A^-$, then $a^+, b^+ \in A^+$. Since $A^- \subseteq G$ and $A^+ \leq H^+$, $a^- \otimes b^- = a^+ \oplus b^+ \in A^+  \subseteq B$ which is again our goal. Next we assume that   $a^+ \in A^+$ and $b^- \in A^- $. Since $A^+ \leq H^+$, $a^+ \oplus b^+ \in A^+$. On the othr hand,  $a^+, b^- \in G$ hence $a^+ \otimes b^- = (a^+ \oplus b^+)^- \in A^- \subseteq B$, as imagined. For the final case, a similar argument proves that if  $a^- \in A^-$ and $b^+ \in A^+$, then  $a^- \otimes b^+$ $=$ $(a^+ \oplus b^+)^- \in A^- \subseteq B$
\end{enumerate}
Hence the result.
\end{proof}


\section{Concluding Remarks}

A gyrogroup of a prime power order $p^n$ is called a $p-$gyrogroup. In this paper, a method for constructing new gyrogroups from the old ones is presented by which it is possible to construct five non-isomorphic $2-$gyrogroups of order $2^n, n \geq 3$. To do this, it is enough to consider the gyrogroups $K(1)$, $L(1)$, $M(1)$, $N(1)$ and $O(1)$ of order eight defined in Tables 1 and 2. Then apply Theorem 2.1 to construct five gyrogroups of order 16 and so on. In the general case, a characterization of the subgyrogroups of new constructed gyrogroup $G$ with respect to the gyrogroup $H^+$ were given. Two open questions for future study were posed.


\begin{table}[H]\label{tab1}
\centering
\caption{The Cayley Tables and associated Gyration Tables of the Gyrogroups $K(1)$, $L(1)$, $M(1)$.}

 \begin{footnotesize}
\begin{tabular}{c|ccccccccc|cccccccc}
$\oplus_K$ & 0 & 1 & 2 & 3 & 4 & 5 & 6 & 7 & $gyr_K$ & 0 & 1 & 2 & 3 & 4 & 5 & 6 & 7\\ \toprule
0        & 0 & 1 & 2 & 3 & 4 & 5 & 6 & 7 & 0     & I & I & I & I & I & I & I & I\\
1        & 1 & 0 & 3 & 2 & 5 & 4 & 7 & 6 & 1     & I & I & I & I & I & I & I & I\\
2&2&3&0&1&6&7&4&5 &  2&I&I&I&I&A&A&A&A\\
3&3&2&1&0&7&6&5&4 & 3&I&I&I&I&A&A&A&A \\
4&4&5&6&7&0&1&2&3  &  4&I&I&A&A&I&I&A&A\\
5&5&4&7&6&1&0&3&2  & 5&I&I&A&A&I&I&A&A\\
6&6&7&4&5&3&2&1&0  &  6&I&I&A&A&A&A&I&I\\
7&7&6&5&4&2&3&0&1  & 7&I&I&A&A&A&A&I&I\\ \bottomrule
\end{tabular}
\begin{tabular}{c|ccccccccc|cccccccc}
$\oplus_L$ & 0 & 1 & 2 & 3 & 4 & 5 & 6 & 7 & $gyr_L$ & 0 & 1 & 2 & 3 & 4 & 5 & 6 & 7\\ \hline
0        & 0 & 1 & 2 & 3 & 4 & 5 & 6 & 7 & 0     & I & I & I & I & I & I & I & I\\
1        & 1 & 0 & 3 & 2 & 5 & 4 & 7 & 6 & 1     & I & I & I & I & I & I & I & I\\
2&2&3&0&1&6&7&4&5 &  2&I&I&I&I&B&B&B&B\\
3&3&2&1&0&7&6&5&4 & 3&I&I&I&I&B&B&B&B \\
4&4&5&6&7&0&1&2&3  &  4&I&I&B&B&I&I&B&B\\
5&5&4&7&6&1&0&3&2  & 5&I&I&B&B&I&I&B&B\\
6&6&7&5&4&3&2&0&1  &  6&I&I&B&B&B&B&I&I\\
7&7&6&4&5&2&3&1&0  & 7&I&I&B&B&B&B&I&I\\ \bottomrule
\end{tabular}
\begin{tabular}{c|ccccccccc|cccccccc}
$\oplus_M$ & 0 & 1 & 2 & 3 & 4 & 5 & 6 & 7 & $gyr_M$ & 0 & 1 & 2 & 3 & 4 & 5 & 6 & 7\\ \hline
0        & 0 & 1 & 2 & 3 & 4 & 5 & 6 & 7 & 0     & I & I & I & I & I & I & I & I\\
1        & 1 & 0 & 3 & 2 & 5 & 4 & 7 & 6 & 1     & I & I & I & I & I & I & I & I\\
2&2&3&0&1&6&7&4&5 &  2&I&I&I&I&C&C&C&C\\
3&3&2&1&0&7&6&5&4 & 3&I&I&I&I&C&C&C&C \\
4&4&5&6&7&1&0&3&2  &  4&I&I&C&C&I&I&C&C\\
5&5&4&7&6&0&1&2&3  & 5&I&I&C&C&I&I&C&C\\
6&6&7&5&4&2&3&1&0  &  6&I&I&C&C&C&C&I&I\\
7&7&6&4&5&3&2&0&1  & 7&I&I&C&C&C&C&I&I\\ \bottomrule
\end{tabular}
\end{footnotesize}
\end{table}

\begin{table}[H]\label{tab4}
\centering
\caption{  The Cayley Tables and associated Gyration Tables of the Gyrogroups  $N(1)$ and $O(1)$.}
\begin{footnotesize}
\begin{tabular}{c|ccccccccc|cccccccc}
$\oplus_N$ & 0 & 1 & 2 & 3 & 4 & 5 & 6 & 7 & $gyr_N$ & 0 & 1 & 2 & 3 & 4 & 5 & 6 & 7\\ \hline
0        & 0 & 1 & 2 & 3 & 4 & 5 & 6 & 7 & 0     & I & I & I & I & I & I & I & I\\
1        & 1 & 0 & 3 & 2 & 5 & 4 & 7 & 6 & 1     & I & I & I & I & I & I & I & I\\
2&2&3&0&1&6&7&4&5 &  2&I&I&I&I&D&D&D&D\\
3&3&2&1&0&7&6&5&4 & 3&I&I&I&I&D&D&D&D \\
4&4&5&6&7&1&0&3&2  &  4&I&I&D&D&I&I&D&D\\
5&5&4&7&6&0&1&2&3  & 5&I&I&D&D&I&I&D&D\\
6&6&7&5&4&3&2&0&1 &  6&I&I&D&D&D&D&I&I\\
7&7&6&4&5&2&3&1&0 & 7&I&I&D&D&D&D&I&I\\ \bottomrule
\end{tabular}
\begin{tabular}{c|ccccccccc|cccccccc}
$\oplus_O$ & 0 & 1 & 2 & 3 & 4 & 5 & 6 & 7 & $gyr_O$ & 0 & 1 & 2 & 3 & 4 & 5 & 6 & 7\\ \hline
0        & 0 & 1 & 2 & 3 & 4 & 5 & 6 & 7 & 0     & I & I & I & I & I & I & I & I\\
1        & 1 & 0 & 3 & 2 & 5 & 4 & 7 & 6 & 1     & I & I & I & I & I & I & I & I\\
2&2&3&0&1&6&7&4&5 &  2&I&I&I&I&E&E&E&E\\
3&3&2&1&0&7&6&5&4 & 3&I&I&I&I&E&E&E&E \\
4&4&5&7&6&1&0&2&3  &  4&I&I&E&E&I&I&E&E\\
5&5&4&6&7&0&1&3&2  & 5&I&I&E&E&I&I&E&E\\
6&6&7&5&4&2&3&1&0  &  6&I&I&E&E&E&E&I&I\\
7&7&6&4&5&3&2&0&1  & 7&I&I&E&E&E&E&I&I\\ \bottomrule
\end{tabular}
\end{footnotesize}
\end{table}

\vskip 3mm

\noindent{\bf Acknowledgement.} The research of the first and second  authors are partially supported by the University of Kashan under grant no. 364988/63.

\end{document}